\documentclass[final]{siamltex}
\usepackage{amsfonts}

\usepackage{graphics}

\def\b#1{{\bf#1}}
\def\x{{\bf X}} \def\l2{\mathbf{L}^2(\Omega)}
\def\c0{\mathbf{H}_0(\mathrm{curl},\Omega) }
\def\d0{\mathbf{H}_0(\mathrm{div0},\Omega;\epsilon) }
\def\rt{\rightarrow} \def\wt{\widetilde} \def\wh{\widehat}
\def\1{C_1} \def\2{C_2} \def\3{C_3} \def\4{C_4} \def\5{C_5} \def\6{C_6} \def\0{C_0}

\usepackage{amssymb}\usepackage{color}
\usepackage{graphicx}

\title{Multigrid methods based on shifted inverse iteration for the Maxwell eigenvalue problem\thanks{Project supported
 by the National Natural Science Foundation of China (Grant No. 11161012). }}

\author{JIAYU HAN\thanks{\tt School of Mathematical Sciences, Guizhou
Normal University, GuiYang, 550001, China (hanjiayu126@126.com). }}
\begin{document}
 \maketitle
\begin{abstract}In this paper  two types of multgrid methods, i.e., the Rayleigh quotient iteration
and the inverse iteration with fixed shift,  are developed
for solving the Maxwell eigenvalue problem with  discontinuous relative magnetic permeability
and electric permittivity. With the aid of the mixed form of source problem associated with the eigenvalue problem, we prove
 the uniform convergence of
 the discrete solution operator  to the solution operator in $\l2$ using discrete compactness of edge element space.
Then we prove the asymptotically optimal error estimates for both multigrid methods. Numerical experiments
confirm our theoretical analysis.
\end{abstract}
\pagestyle{myheadings} \thispagestyle{plain} \markboth{JIAYU HAN }{MULTIGRID METHODS FOR MAXWELL EIGENVALUES}
\section{Introduction}
The Maxwell eigenvalue problem is of basic importance in designing resonant structures for advanced waveguide.
Up to now, the communities from  numerical mathematics and  computational electromagnetism have developed plenty of
numerical methods  for solving this problem (see, e.g., \cite{arbenz,ainsworth,boffi1,boffi3,buffa1,buffa2,brenner,ciarlet1,chen2,chatterjee,milan,jiang,kikuchi,reddy,russo,zhou}).

The difficulty of numerically solving the  eigenvalue problem lies in imposing the divergence-free constraint.
For this purpose, the nodal finite element methods utilize the filter, parameterized and mixed approaches to
find the true eigenvalues \cite{buffa1, ciarlet1}.
The researchers in  electromagnetic field usually   adopt edge finite elements due to   the property of
tangential continuity of electric field \cite{boffi1,chatterjee,hiptmair1,jiang}.
 Using edge finite element methods, when one only considers to compute the nonzero eigenvalue, the divergence-free
 constraint can be dropped from the weak form and satisfied naturally (see \cite{zhou}).
 But this will introduce spurious zero eigenvalues.
Since the eigenspace  corresponding to zero  is infinite-dimensional,
usually, the finer the mesh is, the more the spurious eigenvalues there are.
However, using this form there is no difficulty in computing   eigenvalues on a very coarse mesh.
So the work of \cite{zhou} subtly applies this weak form to two grid method for the Maxwell eigenvalue problem.
That is, they first solve a Maxwell eigenvalue problem on a coarse mesh and then solve a linear Maxwell
equation on a fine mesh.
  Another approach is the mixed form of saddle point type in which a Lagrange multiplier is introduced to impose
  the divergence-free constraint
  (see \cite{kikuchi,arbenz,hiptmair1}).
  A remarkable feature of the mixed form is
   its equivalence to the weak form in \cite{zhou} for nonzero eigenvalues.
   The mixed form is well known as having a good property of no spurious eigenvalues being introduced. However,
   it is not an easy task to   solve it on a fine mesh (see \cite{arbenz,arbenz1}).

The multigrid methods for solving eigenvalue problems originated from the idea of two grid method proposed by \cite{xu2}.
Afterwards, this work was further developed by \cite{hu,yang2,yang3,zhou}. Among them the more recent work   \cite{yang2}
makes a relatively systematical research on  multigrid methods based on shifted inverse iteration especially on
its  adaptive fashion.

Inspired by the above works, this paper is devoted to developing  multigrid methods for solving the Maxwell eigenvalue problem. We first use the mixed form to solve the eigenvalue problem on a coarser mesh and then
solve a series of Maxwell equations on finer and finer meshes without using the mixed form. Roughly speaking,
we develop the two grid method in \cite{zhou} into multigrid method where only  nonzero eigenvalues are focused on.
We prefer to use the mixed form instead of the one in \cite{zhou} on a coarse mesh to capture the information of the true
 eigenvalues.  One reason is that  the mixed form can include the physical zero eigenvalues and rule out the spurious zero ones simultaneously.
 Using it, the physical zero eigenvalues can be captured on a very coarse mesh,   
which is necessary when  the resonant cavity has   disconnected boundaries(see, e.g., \cite{jiang}).
Another reason lies in that the mixed discretization of saddle point type is  not difficult to solve on a coarse mesh.

In this paper, we study two types of multigrid methods based on shifted inverse iteration:
Rayleigh quotient iteration and inverse iteration with fixed shift. The former  is a well-known method
 for solving matrix eigenvalues but the corresponding coefficient matrix is nearly singular and
 difficult to solve to some extend. To overcome this difficulty the latter first performs the
 Rayleigh quotient iteration at previous few steps and then fixes the shift at the following steps as the estimated eigenvalue
 obtained by the former.
 Referring to the error analysis framework in \cite{yang4} and using compactness property of edge element space,
  we first prove the  uniform convergence of
 the discrete solution operator  to the solution operator in $\l2$ and then the error estimates of eigenvalues and eigenfunctions
 for the mixed discretization;
 then we adopt  the analysis tool in \cite{yang2} that is different from the one in \cite{zhou} and  prove the  asymptotically optimal error estimates for both
 multigrid methods.
     In addition, this paper is concerned about the  theoretical analysis   for the case of the
     discontinuous electric permittivity  $\mu$ and magnetic permeability $\epsilon$ in complex matrix form, which has important applications
for the resonant cavity being filled with different dielectric materials invariably.
It is noticed that our multigrid methods and  theoretical results are not only  valid for the lowest order edge elements
but also for high order ones.
More importantly, based on the work of this paper, once given an a posteriori error indicator
of eigenpair one can further develop the adaptive algorithms
of shifted inverse iteration type for the problem.
In the last
section of this paper, we present several numerical examples to validate the efficiency of our methods in different
cases.

Throughout this paper,   we use
the symbol $a \lesssim b$   to mean that $a \le Cb$, where $C$ denotes a positive constant independent of   mesh parameters and iterative times
and may not be the same   in different places.

\section{Preliminaries}
Consider the Maxwell eigenvalue problem in electric field

\begin{eqnarray}\label{s1}
 \mathrm{curl}(\mu^{-1}\mathrm{curl}\mathbf{u})=\omega^2\epsilon\mathbf{u}~~~in~\Omega,\\\label{s2}
 \mathrm{ div}(\epsilon\mathbf{u})=0 ~~~in~\Omega,\\\label{s3}
 \mathbf{u}_t=0 ~~~in~\partial \Omega,
\end{eqnarray}
where $\Omega$ is a bounded Lipschitz polyhedron domain in $\mathbb{R}^d(d=2,3) $, $\mathbf{u}_t$ is the tangential
trace of $\mathbf{u}$.
The coefficient $\mu$ is the electric permittivity, and $\epsilon$ is the magnetic permeability and  piecewise smooth.
    In this paper, $\lambda=\omega^2$ with $\omega$ being the angular frequency, is defined as the eigenvalue of
     this problem. We assume that $\mu,\epsilon$ are two positive definite
      Hermite matrices such that $\mu^{-1},\epsilon\in (L^{\infty}(\Omega))^{d\times d}$
   and there exist  two
positive numbers $ \gamma,\beta$ satifying
 \begin{eqnarray}\label{2.4}
 \overline{\xi}\cdot \mu^{-1}\xi\ge\gamma \overline{\xi}\cdot \xi,~ \overline{\xi}\cdot \epsilon\xi\ge\beta \overline{\xi}\cdot \xi,~\forall 0\neq\xi\in \mathbb{C}^d.
 \end{eqnarray}

\subsection{Some weak forms}
Let
\begin{eqnarray*}
 &&\mathbf{H}_0(\mathrm{curl},\Omega)=\{ \mathbf{u}\in \mathbf L^2(\Omega): \mathrm{curl}(\mathbf{u})\in\mathbf L^2(\Omega),\mathbf{u}_t|_{\partial\Omega}=0\},
\end{eqnarray*}
equipped with the norm $\|\mathbf{u}\|_{\mathrm{curl}}:=\|\mathrm{curl}\mathbf{u}\|_{0}+\|\mathbf{u}\|_{0}$.
Throughout this paper,
$\|\cdot\|_{0}$ and $\|\cdot\|_{0,\epsilon}$ denote the norms in $\l2$ induced by the inner products $(\cdot,\cdot)$
and $(\epsilon\cdot,\cdot)$    respectively.
Define the  divergence-free space:
\begin{eqnarray*}
 &&\mathbf X:=\{\mathbf{u}\in \mathbf H_0(\mathrm{curl},\Omega):~ \mathrm{div}(\epsilon\mathbf{u})=0\}. 
\end{eqnarray*}
The standard weak form of the Maxwell eigenvalue problem (\ref{s1})-(\ref{s3}) is as follows:
Find $(\lambda,\mathbf{u})\in\mathbb{R}\times\mathbf X$ and $\mathbf{u}\neq0$ such that
\begin{eqnarray}\label{2.5}
a(\mathbf{u},\mathbf{v})=\lambda  (\epsilon\mathbf{u},\mathbf{v}), ~~\forall  \mathbf{v}\in\mathbf X,
\end{eqnarray}
where $a(\mathbf{u},\mathbf{v})=(\mu^{-1}\mathrm{curl}\mathbf{u},\mathrm{curl}\mathbf{v})$.
Denote $\|\b v\|_a:=\sqrt{a(\b v,\b v)},\forall \b v \in \c0$.

As the divergence-free space $\x$ in (\ref{2.5}) is difficult to  discretize, alternatively, we would like to solve
the eigenvalue problem (\ref{s1})-(\ref{s3}) in the larger space $\mathbf H_0(\mathrm{curl};\Omega)$, that is: 
 Find $(\lambda,\mathbf{u})\in\mathbb{R}\times \mathbf H_0(\mathrm{curl};\Omega)$ and $\mathbf{u}\neq0$ such that
\begin{eqnarray}\label{s5}
a(\mathbf{u},\mathbf{v})=\lambda  (\epsilon\mathbf{u},\mathbf{v}), ~~\forall  \mathbf{v}\in \mathbf H_0(\mathrm{curl};\Omega).
\end{eqnarray}
Note that when $\lambda\neq0$ (\ref{s5}) and (\ref{2.5}) are equivalent, since (\ref{s5}) implies the divergence-free
condition holds for $\lambda\neq0$ (e.g., see \cite{zhou}).

According to (\ref{2.4}), we have
  \begin{eqnarray}\label{2.7}
  \sqrt\gamma\|\mathrm{curl}\b u\|_{0}\le\|\b u\|_{a}
  \end{eqnarray}
In order to study the  eigenvalue problem in $\c0$ we need the auxiliary bilinear form
$$A(\b w,\b v)=a(\b w,\b v)+ \frac{\gamma}{\beta}(\epsilon\b w,\b v),$$
which defines an equivalent norm $\|\cdot\|_A=\sqrt{A(\cdot,\cdot)}$ in $\c0$.


 By Lax-Milgram Theorem we can define the solution operator $T:\l2\rightarrow \mathbf{X}$ as
\begin{eqnarray}\label{s2.7}
A(T\mathbf{f},\mathbf{v})=   (\epsilon\mathbf{f},\mathbf{v}),~\forall \mathbf{v}\in \mathbf{X}.
\end{eqnarray}

Then  the eigenvalue problem (\ref{2.5})   has the operator form
$$T\b u=\wt\lambda^{-1}\b u~with~\wt\lambda=\lambda+\frac{\gamma}{\beta}.$$


The following mixed weak form of saddle point type   can be found in \cite{kikuchi,arbenz,arbenz1,boffi3,jiang}:
Find $(\lambda,\mathbf{u},\sigma)\in \mathbb{R}\times   \c0\times H^1_0(\Omega)$ with
$\mathbf{u}\ne0$  such that
\begin{eqnarray}
a(\mathbf{u},\mathbf{v})+\overline{b(\mathbf{v},\sigma)}&=&\lambda(\epsilon\mathbf{u},\mathbf{v}),~~\forall \mathbf{v}\in \c0,\label{3.1ss}\\
b(\mathbf{u},p)&=&0,~~\forall p\in H^1_0(\Omega),\label{3.2ss}
\end{eqnarray}
where $b(\mathbf{v}, p):=(\epsilon\b v,\nabla p)$ for any $\b v\in\c0,~p\in H^1_0(\Omega)$.

We introduce the corresponding mixed equation: Find $\widetilde{T}\b f\in \c0$ and $ S \b f\in  H^1_0(\Omega)$ for $\b f \in \l2$ such that
\begin{eqnarray}\label{s2.3}
A(\widetilde{T}\b f,\mathbf{v})+\overline{b(\mathbf{v}, S \b f)}&=& (\epsilon\mathbf{f},\mathbf{v}),~~\forall \mathbf{v}\in \c0,\\\label{s2.4}
b(\widetilde{T}\mathbf{f},p)&=&0,~~\forall p\in  H^1_0(\Omega);
\end{eqnarray}
the following LBB condition
can be verified by taking $\b w=\nabla \b v$
\begin{eqnarray*}
\sup_{\b w\in\c0}\frac{|b(\b w,\b v)|}{\|\b w\|_{\mathrm{curl}}}\ge \beta |\b v|_{1},~\forall \b v\in H^1_0(\Omega).
\end{eqnarray*}
This yields the existence and uniqueness of linear bounded operators $\wt T$ and $S$ (see \cite{brezzi}).
Due to Helmholtz decomposition $\mathbf{H}_0(\mathrm{curl};\Omega)=\nabla {H}_0^1(\Omega)\bigoplus \mathbf X$, it is easy to
see $R(T)\subset\x$ and $S \b f=0$, $T\b f=\widetilde{T}\b f$  for any $\b f\in\x$. Hence  $T$ and $\widetilde{T}$
share the same eigenpairs. More importantly,  the operator $\wt T: \l2\rt\l2$ is self-adjoint. In fact,
$\forall \mathbf{v},\mathbf{w}\in \l2$,
\begin{eqnarray}\label{s2.6}
(\epsilon \b w,\wt T\b v)=A(\wt T\mathbf{w},\wt T\mathbf{v})
=\overline{A(\wt T\mathbf{v},\wt T\mathbf{w})}
=\overline{(\epsilon\mathbf{v},\wt T\mathbf{w})}= (\epsilon\wt T\mathbf{w},\mathbf{v}).
\end{eqnarray}
Note that $\wt T$ is compact as a operator from $\l2$ to $\l2$ and from
$\mathbf{X}$ to $ \mathbf{X}$ since $\mathbf{X}\hookrightarrow\l2$ compactly (see Corollary 4.3 in \cite{hiptmair1}).


\subsection{Edge element discretizations and error estimates}
 We will consider the edge element approximations based on the  weak forms (\ref{2.5}),
 (\ref{s5}) and (\ref{3.1ss})-(\ref{3.2ss}). Let $\pi_h$ be a shaped-regular triangulation of $\Omega$ composed of the elements ${\kappa}$.
Here we restrict our attention to  edge   elements on tetrahedra because the argument for edge   elements on hexahedra
is the same.
The $k$($k\ge0$) order edge element of the first family \cite{nedelec} generates the space
$$\mathbf{V}_h=\{\mathbf{u}_h\in\mathbf H_0(\mathrm{curl},\Omega):\mathbf{u}_h|_\kappa\in [P_\kappa(k)]^d\bigoplus\mathbf{x}\times [\widetilde{P}_\kappa(k)]^d\},$$
where $P_\kappa(k)$ is the polynomial space of  degree less than or equal to $ k$ on $\kappa$, ${\widetilde{P}}_\kappa(k)$ is the
homogeneous polynomial space of  degree $ k$  on $\kappa$, and $\b x =(x_1,\cdots,x_d)^T$.
We also introduce the discrete divergence-free space
$$\mathbf X_h=\{\mathbf{u}_h\in\mathbf V_h: (\epsilon\mathbf{u}_h,\nabla p)=0,~~ \forall p \in U_h\},$$
where $U_h$ is the standard  Lagrangian finite element space vanishing on $\partial\Omega$ of total degree less than or
equal to $k+1$ and $\nabla U_h\subset \mathbf{V}_h.$ 

The standard finite element discretization of (\ref{2.5}) is stated as: Find $(\lambda_h,\mathbf{u}_h)\in R\times\b \x_h$ and $\mathbf{u}_h\neq0$ such that
\begin{eqnarray}\label{2.13}
a(\mathbf{u}_h,\mathbf{v}_h)=\lambda_h  (\epsilon\mathbf{u}_h,\mathbf{v}_h),~~\forall \mathbf{v}_h\in\mathbf \x_h.
\end{eqnarray}
It is also equivalent to the following form for  nonzero $\lambda_h$ (see \cite{zhou}):  
Find $(\lambda_h,\mathbf{u}_h)\in \mathbb{R}\times\b V_h$ and $\mathbf{u}_h\neq0$ such that
\begin{eqnarray}\label{2.14}
a(\mathbf{u}_h,\mathbf{v}_h)=\lambda_h  (\epsilon\mathbf{u}_h,\mathbf{v}_h),~~\forall \mathbf{v}_h\in\mathbf V_h.
\end{eqnarray}

In order to investigate the convergence of edge element discretization   (\ref{2.13}), we have to study the convergence of edge element discretization for
the associated Maxwell source problem. 
 Then by Lax-Milgram Theorem we can define the solution operator $T_h:\l2
\rightarrow \mathbf{X}_h$ as
\begin{eqnarray}\label{s2.5}
A(T_h\mathbf{f},\mathbf{v})=   (\epsilon\mathbf{f},\mathbf{v}),~\forall \mathbf{v}\in \mathbf{X}_h.
\end{eqnarray}
Then  the eigenvalue problem (\ref{2.13})  has the operator form
$$T_h\b u_h=\wt\lambda_h^{-1}\b u_h~with~\wt \lambda_h=\lambda_h+\frac{\gamma}{\beta}.$$
Introduce the  discrete form of (\ref{3.1ss})-(\ref{3.2ss}):
Find $(\lambda_h,\mathbf{u}_{h},\sigma_h)\in \mathbb{R}\times \b V_h\times U_h$,
$\mathbf{u}_{h}\ne0$ such that
\begin{eqnarray}
a(\mathbf{u}_{h},\mathbf{v})+\overline{b(\mathbf{v},\sigma_h)}&=&\lambda_h(\epsilon\mathbf{u}_{h},\mathbf{v}),
~~\forall \mathbf{v}\in \b V_h,\label{3.1s}\\
b(\mathbf{u}_{h},p)&=&0,~~\forall p\in U_h.\label{3.2s}
\end{eqnarray}
Introduce the corresponding operators: Find $\widetilde{T}_{h}\b f\in \b V_h$ and $S_{h}\b f \in   U_h$ for any $\b f \in\l2$
\begin{eqnarray}\label{s2.1}
A(\widetilde{T}_{h}\b f,\mathbf{v})+\overline{b(\mathbf{v},S_{h}\b f)}&=& (\epsilon\mathbf{f},\mathbf{v}),~~\forall \mathbf{v}\in\b V_h,\\\label{s2.2}
b(\widetilde{T}_{h}\b f,p)&=&0,~~\forall p\in U_h.
\end{eqnarray}
Due to discrete Helmholtz decomposition $\mathbf{V}_h=\nabla U_h\bigoplus \mathbf X_h$, it is easy to know $R(T_h)\subset\x_h$ and $S_h\b f=0$, $T_h\b f=\widetilde{T}_h\b f$ for any $\b f\in\x+\x_h$.
Hence $T_h$ and $\widetilde{T}_h$ share the same eigenpairs.

 Similar to (\ref{s2.3})-(\ref{s2.4}), one can  verify the corresponding LBB condition for the discrete mixed  form  (\ref{s2.1})-(\ref{s2.2}).
 According to the theory of mixed finite elements (see \cite{brezzi}), we get for all $\b f\in\l2$,
\begin{eqnarray}\label{15}
\|\wt T_h \mathbf f \|_{A}+\|\wt T \mathbf f \|_{A}+|S_h\b f|_1+|S\b f|_1\le \1 \| \mathbf f\|_{0,\epsilon},
\end{eqnarray}
\begin{eqnarray}\label{s2.18}
\|\wt T \mathbf f-\wt T_h\mathbf{f}\|_{A}+|S \b f-S_h \b f|_1\le \2(\inf\limits_{\mathbf{v}_h\in\mathbf V_h}\|\widetilde{T} \mathbf f-\mathbf{v}_h\|_{\mathrm{curl}}+\inf\limits_{v_h\in  U_h}|S \b f-v_h|_1).
\end{eqnarray}
Similar to (\ref{s2.6}) we can prove $\wt T_h:\l2\rt\l2$ is self-adjoint in the sense of $(\epsilon\cdot,\cdot)_{0}$. In fact, $\forall \b w,\b v\in\l2$,
\begin{eqnarray*}
(\epsilon \b w,\wt T_h\b v)=A(\wt T_h \b w,\wt T_h\b v)=\overline{A(\wt T_h \b v,\wt T_h\b w)}=
\overline{(\epsilon \b v,\wt T_h\b w)}= (\epsilon \wt T_h \b w,\b v).
\end{eqnarray*}

 The discrete compactness is a very interesting and important property in edge elements because it is intimately related to the property of the collective compactness. Kikuchi \cite{kikuchi1} first successfully applied this property to numerical analysis of electromagnetic problems,
 and more recently it was further developed by \cite{boffi2,boffi3,caorsi,kirsch,monk2} and so on.
  The following lemma, which states the discrete compactness of $\x_h$ into $\l2$, is a direct citation of Theorem 4.9 in \cite{hiptmair1}.
\begin{lemma}(Discrete compactness property) Any sequence $\{\b v_h\}_{h>0}$ with $\b v_h\in \x_h$ that is uniformly bounded in $\b H(\mathrm{curl},\Omega)$ contains a subsequence that converges strongly in $\b L^2(\Omega)$.
\end{lemma}

In the remainder of this subsection, we will prove the error estimates for the discrete forms (\ref{3.1s})-(\ref{3.2s}),
(\ref{2.13}) or (\ref{2.14}) with $\lambda_h\ne0$.
The authors in \cite{yang4} have built a general analysis framework for
the a priori error estimates of mixed form (see Theorem 2.2 and Lemma 2.3 therein). Although we cannot directly apply their theoretical results to
the mixed discretization (\ref{3.1s})-(\ref{3.2s}), we can use its proof idea
to derive the following Lemma 2.2 and Theorem 2.3.
The following uniform convergence provides us with the possibility
 to use the spectral approximation theory in \cite{babuska}.
\begin{lemma} There holds the uniform convergence
\begin{eqnarray*}
&\|\widetilde{T}-\widetilde{T}_h\|_{\l2}\rightarrow0,~h\rightarrow0.
\end{eqnarray*}
\end{lemma}
\begin{proof}
Since $\cup_{h>0}\b V_h$ and $\cup_{h>0}  U_h$ are dense in $\c0$ and $ H^1_0(\Omega)$, respectively, we deduce from (\ref{s2.18}) for any $\b f\in \l2$
\begin{eqnarray*}
\|\wt T \mathbf f-\wt T_h\mathbf{f}\|_{A} \le \2(\inf\limits_{\mathbf{v}_h\in\mathbf V_h}\|\widetilde{T} \mathbf f-\mathbf{v}_h\|_{\mathrm{curl}}+\inf\limits_{v_h\in  U_h}|S\b f-v_h|_1)\rt 0.
 \end{eqnarray*}
 That is,  $\wt T_h$ converges to $\wt T$ pointwisely.
 Since $\wt T,\wt T_h : \l2\rt\c0$ are linear bounded uniformly with respect to $h$, $\cup_{h>0}(\wt T-\wt T_h)B$ is a bounded set in $\c0$ where $B$ is the unit ball in $\l2$. From
$\x\hookrightarrow \l2$ compactly and the discrete compactness property of $\x_h$ in Lemma 2.1, we
know that $\cup_{h>0}(\wt T-\wt T_h)B$ is a relatively compact set in $\l2$, which implies  collectively
compact convergence $\wt T_h\rt\wt T$.  Noting $\wt T, \wt T_h: \l2\rt \l2$ are self-adjoint,   due to Proposition 3.7 or Table 3.1 in \cite{chatelin} we get  $\|\widetilde{T}-\widetilde{T}_h\|_{\l2}\rightarrow0,~h\rightarrow0.$
 This ends the proof.
\end{proof}
Prior to proving the error estimates for   edge element discretizations, we define some notations as follows.
Let $\lambda$ be the $k$th eigenvalue of (\ref{2.5}) or (\ref{3.1ss})-(\ref{3.2ss}) of multiplicity $q$.
Let $\lambda_{j,{h}}~(j=k,k+1,\cdots,{k+q-1})$ be eigenvalues of $T_{h}$ that converge to the eigenvalue $\lambda=\lambda_k=\cdots=\lambda_{k+q-1}$.
Here and hereafter we use ${M}(\lambda)$ to denote the space spanned by all eigenfunctions  corresponding to the eigenvalue $\lambda$,
and ${M_h}(\lambda)$ to denote the direct sum of all eigenfunctions corresponding to the eigenvalues $\lambda_{j,{h}}~(j=k,k+1,\cdots,{k+q-1})$.
For argument convenience,
hereafter we denote $\wt\lambda_{j}=\lambda_{j}+\frac{\gamma}{\beta}$ and
 $\wt\lambda_{j,h}=\lambda_{j,h}+\frac{\gamma}{\beta}$.
 Now we introduce the following small quantity:
\begin{eqnarray*}
 &&\delta_h(\lambda)=\sup_{\mathbf{u}\in {M}(\lambda),\atop\|\b u\|_{A}=1} \inf_{\mathbf{v}\in \b V_{h}}\|\mathbf{u}-\mathbf{v}\|_{A}.
\end{eqnarray*}
Thanks to (\ref{s2.18}) and (\ref{2.7}) we have
\begin{eqnarray}\label{2.21}
 \wt\lambda\|(  \wt T_{h}-  \wt T)|_{M(\lambda)}\|_{A}\le\2\wt\lambda \sup_{\b u\in {M}(\lambda),\atop \|\b u\|_{A}=1}
 \inf_{\b v_h\in\b V_h}\|  T\b u-  \b v_h\|_{\mathrm{curl}}\le\2\sqrt\gamma\delta_{h}(\lambda).
\end{eqnarray}

The error estimates of edge elements for  the Maxwell eigenvalue problem have been obtained in, e.g., \cite{boffi1,boffi3,monk2,russo}. Here we would like to use the quantity $\delta_h(\lambda)$ to characterize the error
for eigenpairs.
From the  spectral approximation, we actually derive the a priori error estimates for the discrete eigenvalue
problem  (\ref{2.14}) with $\lambda_h\ne0$, (\ref{2.13}) or (\ref{3.1s})-(\ref{3.2s}).
\begin{theorem}
Let $\lambda$  be the  eigenvalue of (\ref{2.5}) or (\ref{3.1ss})-(\ref{3.2ss}) and let  $\lambda_h$ be the discrete eigenvalue of
(\ref{2.13}) or (\ref{3.1s})-(\ref{3.2s}) converging to
$\lambda$.
There exist   $h_0 > 0$ such that if $h \le h_0$ then
for any eigenfunction $\b u_h$ corresponding to $\lambda_h$ with $\|\b u_h\|_A = 1$ there exists $\b u \in
M(\lambda)$ such that
\begin{eqnarray}\label{2.22}
 &&\|\b u-\b u_h\|_{A}\le\3\delta_h(\lambda)
\end{eqnarray}
and for any $\b u\in M(\lambda)$ with  $\|\b u\|_A=1$ there exists $\b u_h\in M_h(\lambda)$ such that
\begin{eqnarray}\label{2.23}
 &&\|\b u-\b u_h\|_{A}\le\3\delta_h(\lambda),
\end{eqnarray}
where the positive constant $\3$ is independent of   mesh parameters.
\end{theorem}
\begin{proof} We take $\lambda=\lambda_k$.
Suppose  $\b u_h$ is an eigenfunction of (\ref{3.1s})-(\ref{3.2s}) corresponding to $\lambda_h$ satisfying
$\|\b u_h\|_A=\sqrt{\wt \lambda_h}\|\b u_h\|_{0,\epsilon} = 1$.
Then according to Theorems 7.1 and 7.3 in \cite{babuska} and Lemma 2.2 there exists $\b u\in M(\lambda)$ satisfying
\begin{eqnarray}\label{2.28}
&&\|\b u_h-\b u\|_{0,\epsilon}\lesssim \|(\wt T-\wt T_h)|_{M(\lambda)}\|_{0,\epsilon},\\
&&|\lambda_{j,h}-\lambda|\lesssim \|(\wt T-\wt T_h)|_{M(\lambda)}\|_{0,\epsilon}~for~j=k,\cdots,k+q-1.\label{2.29}
\end{eqnarray}
By a simple calculation, we deduce
\begin{eqnarray*}
|\|\b u_h-\b u\|_A- \|\wt \lambda(\wt T-\wt T_h)\b u\|_A|&=&|\|\wt \lambda_h \wt T_h\b u_h-\wt \lambda \wt T\b u\|_A-
\|\wt \lambda(\wt T-\wt T_h)\b u\|_A|\\
&\le&\|\wt T_h(\wt \lambda_h\b u_h-\wt \lambda \b u)\|_A\\
&\lesssim &\|\wt \lambda_h\b u_h-\wt \lambda \b u\|_{0,\epsilon}\\
&\lesssim &|\lambda_h-\lambda|\b\|\b u_h\|_{0,\epsilon}+\wt \lambda\|\b u_h- \b u\|_{0,\epsilon}.
\end{eqnarray*}
Since the equality  (\ref{2.21}) implies $\|(\wt T-\wt T_h)|_{M(\lambda)}\|_{0,\epsilon}\lesssim\delta_h(\lambda)$,
this together with (\ref{2.28})-(\ref{2.29})  yields (\ref{2.22}).
Conversely, suppose $\b u$ is an eigenfunction of (\ref{3.1ss})-(\ref{3.2ss}) corresponding to $\lambda$
 satisfying $\|\b u\|_A =\sqrt{\wt \lambda}\|\b u\|_{0,\epsilon} =  1$. Then according to    Theorems 7.1 in \cite{babuska} and Lemma 2.1 there exists $\b u_h\in M_h(\lambda)$ satisfying
 \begin{eqnarray}\label{2.30}
 \|\b u_h-\b u\|_{0,\epsilon}\lesssim \|(\wt T-\wt T_h)|_{M(\lambda)}\|_{0,\epsilon}.
 \end{eqnarray}
 Let $\b u_h=\sum_{j=k}^{k+q-1} \b{\overline{u}}_{j,h}$ where
 $\b{\overline{u}}_{j,h}$  is the eigenfunction corresponding to $\lambda_{j,h}$ such that
$\{\b{\overline{u}}_{j,{h}}\}_{j=k}^{k+q-1}$ constitutes an orthogonal basis of $M_{h}(\lambda)$ in
$(\epsilon\cdot,\cdot)$. Then
\begin{eqnarray*}
&&|\|\b u_h-\b u\|_A- \|\wt\lambda(\wt T-\wt T_h)\b u\|_A|\le\|\b u_h-\wt T_h(\wt \lambda \b u)\|_A\\
&&~~~\lesssim \|\wt T_h(\sum_{j=k}^{k+q-1} \wt \lambda_{j,h}\b{\overline{u}}_{j,h}-\wt\lambda \b u)\|_{A}\\
&&~~~\lesssim \| \sum_{j=k}^{k+q-1} (\wt\lambda_{j,h}-\wt\lambda_h)\b{\overline{u}}_{j,h}+\wt\lambda_h\b u_h-\wt\lambda \b u\|_{0,\epsilon}.
\end{eqnarray*}
Since  $\|(\wt T-\wt T_h)|_{M(\lambda)}\|_{0,\epsilon}\lesssim\delta_h(\lambda)$,
this together with (\ref{2.29})-(\ref{2.30}) yields (\ref{2.23}).
\end{proof}

\indent{\bf Remark 2.1.} Based on the estimate (\ref{2.22}), one can naturally obtain the optimal convergence order
$\mathcal{O}(\delta_h^2(\lambda))$ for  $|\lambda_h-\lambda|$ using the Rayleigh quotient relation (\ref{s2.12})
in the following section.
In addition, note that when $\lambda\ne0$ in Theorem 2.2 the estimate (\ref{2.22}) implies $\|\b u_h\|_a$
converges to  $\|\b u\|_a=\sqrt{\lambda}\|\b u\|_{0,\epsilon}>0$.
Here we introduce $\b{\wh u}_h= \frac{\b u_h}{\|\b u_h\|_a}$ then $\b{u}_h= \frac{\b{\wh u}_h}{\|\b{\wh u}_h\|_A}$ and
(\ref{2.22}) gives
\begin{eqnarray}\label{2.24}
 &&\|\b u-\b{\widehat{u}}_h\|_{A}\le\3\delta_h(\lambda).
\end{eqnarray}
For simplicity of notation, we still use the same $C_3$ and $\b u$ in the above estimate as in (\ref{2.22})-(\ref{2.23}).

\indent{\bf Remark 2.2.}
When $\Omega$ is a Lipschitz polyhedron and $\epsilon,\mu$ are properly smooth, it is known that
$\x\subset \b (H^\sigma(\Omega))^3$ ($\sigma\in(1/2,1]$) (see \cite{girault,boffi2,monk}) and
 $\delta_h(\lambda)\lesssim h^\sigma$. In particular, if $\x\subset  \{\b v\in\b H^s(\Omega):\mathrm{curl}\b v\in \b H^s(\Omega)\}$ ($1\le s\le k+1$) then
$\delta_h(\lambda)\lesssim h^s $ (see  Theorem 5.41 in \cite{monk}).
\section{Multigrid schemes based on shifted inverse iteration}
\label{sec:2}
\subsection{Multigrid Schemes}
In practical computation, the information on the physical zero eigenvalues can be easily captured on a coarse mesh $H$
 using  the mixed discretization (\ref{3.1s})-(\ref{3.2s}).
In this section we shall present our multigrid methods for solving nonzero Maxwell eigenvalue.
The following schemes are proposed by  \cite{yang2,yang3}. Note that we assume in the following schemes the numerical eigenvalue
$\lambda_H$ approximates the nonzero eigenvalue $\lambda$.\\
\textbf{Scheme 3.1.~~} Rayleigh quotient iteration.\\
Given the maximum number of iterative times $l$.\\
\textbf{Step 1}. Solve the eigenvalue problem (\ref{s1})-(\ref{s3}) on coarse finite element space $\b V_H\times U_H$: find $(\lambda_H,\mathbf{u}_{H},\sigma_H)\in R\times \b V_H\times U_H$,
$\|\mathbf{u}_{H}\|_a=1$ such that
\begin{eqnarray*}
a(\mathbf{u}_{H},\mathbf{v})+\overline{b(\mathbf{v},\sigma_H)}&=&\lambda_H(\epsilon\mathbf{u}_{H},\mathbf{v}),~~\forall \mathbf{v}\in \b V_H,\\
b(\mathbf{u}_{H},p)&=&0,~~\forall p\in U_H.
\end{eqnarray*}
\textbf{Step 2}. $\mathbf{u}^{h_0}\Leftarrow \mathbf{u}_{H},~\lambda^{h_0}\Leftarrow \lambda_H,~i\Leftarrow 1$.\\
\textbf{Step 3}. Solve an equation on $\b V_{h_{i}} $: find $(\mathbf{u}',\sigma')\in \b V_{h_i}$ such that
\begin{eqnarray*}
a(\mathbf{u'},\mathbf{v})-\lambda^{h_{i-1}}(\epsilon\mathbf{u'},\mathbf{v})&=&(\epsilon\mathbf{u}^{h_{i-1}},\mathbf{v}),~~\forall \mathbf{v}\in \b V_{h_i}.
\end{eqnarray*}
\indent Set 
$\mathbf{u}^{h_i}=\mathbf{u'}/\|\mathbf{u'}\|_a .$
\\
\textbf{Step 4}. Compute the Rayleigh quotient
\begin{eqnarray*}
\lambda^{h_i}=\frac{a(\mathbf{u}^{h_{i}},\mathbf{u}^{h_{i}})}{(\epsilon\mathbf{u}^{h_{i}},\mathbf{u}^{h_{i}})}.
\end{eqnarray*}
\textbf{Step 5}. If $i=l$, then output $(\lambda^{h_{i}},\mathbf{u}^{h_{i}})$, stop; else, $i\Leftarrow i+1$, and return to step 3.\\

In Step 3 of the above Scheme, when the shift $\lambda^{h_{l-1}}$ is close to the exact eigenvalue enough,
the coefficient matrix of linear equation is nearly singular. Hence the following algorithm gives a natural
way of handling this problem.

\textbf{Scheme 3.2.~~} Inverse iteration with fixed shift.\\
Given the maximum number of iterative times $l$ and $i0$.\\
\textbf{Step 1-Step 4}. The same as Step 1-Step 4 of Scheme 3.1.
\\
\textbf{Step 5}. If $i>i0$ then $i\Leftarrow i+1$ and return to Step 6; else $i\Leftarrow i+1$ and return to Step 3.\\
\textbf{Step 6}. Solve an equation on $\b V_{h_{i}} $: find $(\mathbf{u}',\sigma')\in \b V_{h_i}$ such that
\begin{eqnarray*}
a(\mathbf{u'},\mathbf{v})-\lambda^{h_{i0}}(\epsilon\mathbf{u'},\mathbf{v})&=&(\epsilon\mathbf{u}^{h_{i-1}},\mathbf{v}),~~\forall \mathbf{v}\in \b V_{h_i}.
\end{eqnarray*}
\indent Set 
$\mathbf{u}^{h_i}=\mathbf{u'}/\|\mathbf{u'}\|_a$.
\\
\textbf{Step 7}. Compute the Rayleigh quotient
\begin{eqnarray*}
\lambda^{h_i}=\frac{a(\mathbf{u}^{h_{i}},\mathbf{u}^{h_{i}})}{(\epsilon\mathbf{u}^{h_{i}},\mathbf{u}^{h_{i}})}.
\end{eqnarray*}
\textbf{Step 8}. If $i=l$, then output $(\lambda^{h_{i}},\mathbf{u}^{h_{i}})$, stop; else, $i\Leftarrow i+1$, and return to step 6.\\

\indent{\bf Remark 3.1.} The mixed discretization (\ref{3.1s})-(\ref{3.2s}) was adopted by the literatures
 \cite{kikuchi,arbenz,jiang}. As is proved in Theorem 2.2, using this discretization  we can compute the Maxwell eigenvalues
 without introducing    spurious eigenvalues. However, it is also a saddle point problem that is difficult
 to solve on a fine mesh (see \cite{arbenz,arbenz1}). Therefore, the multigrid schemes can properly overcome the difficulty since
 we only solve (\ref{3.1s})-(\ref{3.2s})
  on a coarse mesh, as shown in step 1 of Schemes 3.1 and 3.2. Moreover, in order to further improve the
   efficiency of solving the equation in Steps 3 and 6 in Schemes 3.1 and 3.2 the HX preconditioner in
   \cite{hiptmair2} is a good choice (see \cite{zhou}).


\subsection{Error Analysis}
In  this subsection, we aim to prove the error estimates for Schemes 3.1 and 3.2.
We shall analyze the constants in the error estimates are independent of mesh parameters and iterative times $l$.
First of all, we give
two useful lemmas.
\begin{lemma}\label{l1}
For any nonzero $\mathbf{u,v}\in \c0,$   there hold
\begin{eqnarray}\label{s3.1}
\|\frac{\mathbf{u}}{\|\mathbf{u}\|_A}-\frac{\mathbf{v}}{\|\mathbf{v}\|_A}\|_A\leq2\frac{\|\mathbf{u-v}\|_A}{\|\mathbf{u}\|_A},~~~
~\|\frac{\mathbf{u}}{\|\mathbf{u}\|_A}-\frac{\mathbf{v}}{\|\mathbf{v}\|_A}\|_A\leq2\frac{\|\mathbf{u-v}\|_A}{\|\mathbf{v}\|_A}.
\end{eqnarray}
\end{lemma}
\begin{proof}See \cite{yang3}.\end{proof}

\begin{lemma}\label{l2}
Let $(\lambda,\b u)$ be an eigenpair of
(\ref{2.5}) or of (\ref{s5}) with $\lambda\ne0$, then for any $v \in \c0\backslash\{0\}$, the
Rayleigh quotient $R(\b v)=\frac{a(\b v,\b v)}{\|\b v\|_{0,\epsilon}^{2}}$ satisfies
\begin{eqnarray}\label{s2.12}
R(\b v)-\lambda=\frac{\|\b v-\b u\|_{a}^{2}}{\|\b v\|_{0,\epsilon}^{2}}-\lambda
\frac{\|\b v-\b u\|_{0,\epsilon}^{2}}{\|\b v\|_{0,\epsilon}^{2}}.
\end{eqnarray}
\end{lemma}
\begin{proof}
See pp.699 of \cite{babuska}.
\end{proof}

 The basic relation in Lemma 3.2 cannot be directly applied to our theoretical analysis, so in the following  we shall
further simplify   the estimate (\ref{s2.12}).
 Let $\overline{C}= (\frac{\gamma}{\beta})^{1/2}$ then according to the definition of $A(\cdot,\cdot)$,
\begin{eqnarray*}
\|\b v\|_{0,\epsilon}\leq \overline{C}^{-1}\|\b v\|_{A},~~~\forall \b v\in \c0.
\end{eqnarray*}
If $\b u\in M(\lambda)$, $\b v_h\in \b V_h$, $\|\b v_h\|_{A}=1$ and
$\|\b v_h-\b u\|_{A}\leq  \overline{C}(4\sqrt{\wt\lambda})^{-1}$, then by Lemma 3.1 we deduce
\begin{eqnarray*}
&&\|\b v_h-\frac{\b u}{\|\b u\|_{A}}\|_{A}\leq 2\|\b v_h-\b u\|_{A}\leq
 \overline{C}(2\sqrt{\wt\lambda})^{-1},\nonumber\\
&&\|\b v_h-\frac{\b u}{\|\b u\|_{A}}\|_{0,\epsilon}\leq
\overline{C}^{-1}\|\b v_h-\frac{\b u}{\|\b u\|_{A}}\|_{A}\leq (2\sqrt{\wt\lambda})^{-1},
\end{eqnarray*}
which together with
$\|\b u\|_{A}=\sqrt{\wt\lambda}\| \b u \|_{0,\epsilon}$  yields
\begin{eqnarray*}
\|\b v_h\|_{0,\epsilon}\geq
 \frac{\|\b u\|_{0,\epsilon}}{\|\b u\|_A}-\|\b v_h-\frac{\b u}{\|\b u\|_{A}}\|_{0,\epsilon}
 \geq
(2\sqrt{\wt\lambda})^{-1}.
\end{eqnarray*}
Hence, from Lemma \ref{l2} we get the following estimate
\begin{eqnarray}\label{3.6}
|R(\b v_h)-\lambda|\leq C_{4}\|\b v_h-\b u\|_{A}^{2},
\end{eqnarray}
where 
$C_{4}=4\wt\lambda(1+\lambda\overline{C}^{-2}).$
Define the    operators $\wh T: \l2\rt\c0$ and $\widehat{T}_h:\l2
\rightarrow \mathbf{V}_h$  as
\begin{eqnarray}\label{3.4}
&&A(\widehat{T}\b f,\b v)=  (\epsilon\b f,\b v),~\forall \b v\in \c0,\\\label{3.5}
&&A(\widehat{T}_h\mathbf{f},\mathbf{v}_h)=   (\epsilon\mathbf{f},\mathbf{v}_h),~\forall \mathbf{v}_h\in \mathbf{V}_h.
\end{eqnarray}
The following lemma   turns our attention from the spectrum  of $T$ and $T_h$
to that of $\wh T$  and $\wh T_h$.
 \begin{lemma} $T$, $\wt T$ and $\wh T$ share the    eigenvalues   greater than $\frac{\gamma}{\beta}$
 and the associated eigenfunctions.
 The same conclusion is valid for   $T_h$, $\wt T_h$ and $\wh T_h$.
 Moreover, $\wh T|_{\x}=\wt T|_{\x} =T|_{\x}$ and
 $\wh T_{h}|_{\x}=\wt T_{h}|_{\x}  =T_{h}|_{\x}$.
  \end{lemma}
\begin{proof} The assertions  regarding the relations among $T$, $\wt T$, $T_h$ and $\wt T_h$ have been described in section 2.
Next we shall prove the relations between $T$ and  $\wt T$ and between $T_h$ and $\wt T_h$. By the definition of $T$ and
$\wh T$, the eigenpair ($\wt\lambda,\b u$) of $T$ satisfies
$A(\mathbf{u},\mathbf{v})=   \wt\lambda(\epsilon\mathbf{u},\mathbf{v})~for~all~\mathbf{v}\in \mathbf{X}$
and the eigenpair ($\wt\lambda,\b u$) of $\wt T$ satisfies
$A(\mathbf{u},\mathbf{v})=   \wt\lambda(\epsilon\mathbf{u},\mathbf{v})~for~all~\mathbf{v}\in \c0$.
Note that the above two weak forms are equivalent when $\wt\lambda>\frac{\gamma}{\beta}$ (since this implies
the eigenfunction $\b u$ of the latter satisfies   divergence-free constraint).
Hence $T$ and $\wh T$ share the    eigenvalues  $\wt\lambda>\frac{\gamma}{\beta}$
 and the associated eigenfunctions. Similarly one can check $T_h$ and $\wh T_h$ share the    eigenvalues
 $\wt\lambda_h>\frac{\gamma}{\beta}$
 and the associated eigenfunctions.
 Thanks to Helmholtz decomposition $\c0=\nabla {H}_0^1(\Omega)\bigoplus \mathbf X$ and (\ref{s2.7}), we also have for any $\b f\in\x$
\begin{eqnarray*}
A(T\mathbf{f},\mathbf{v})=(\epsilon\mathbf{f},\mathbf{v}),~\forall \mathbf{v}\in \c0.
\end{eqnarray*}
This together with (\ref{3.4}) yields $T|_\x=\wh T|_\x$.
 Thanks to discrete Helmholtz decomposition $\mathbf{V}_h=\nabla U_h\bigoplus \mathbf X_h$ and (\ref{s2.5}), we also have for any $\b f\in\x_h+\x$
\begin{eqnarray*}
A(T_h\mathbf{f},\mathbf{v})=   (\epsilon\mathbf{f},\mathbf{v}),~\forall \mathbf{v}\in \mathbf{V}_h.
\end{eqnarray*}
This together with (\ref{3.5}) yields $T_h|_\x=\wh T_h|_\x$.
\end{proof}

Denote $dist(\mathbf{w},W)=\inf\limits_{\mathbf{v}\in W}\|\mathbf{w}-\mathbf{v}\|_A$.
 For  better understanding of notations, hereafter
 we write $\nu_k=\wt\lambda^{-1},\nu_{j,h}=\wt\lambda_{j,h}^{-1}$, and
 $M_h(\nu_k)=M_h(\lambda_k)$.\\
\indent The following lemma  (see \cite{yang2}) is valid since $T_h$ and $\wh T_h$ share the same eigenpairs.
It provides  a crucial   tool for  analyzing the error of multigrid Schemes 3.1 and 3.2.

\begin{lemma}\label{l4}
Let $(\nu_0,\mathbf{u}_0)$ is an approximate eigenpair of $(\nu_k,\mathbf{u}_k)$, where $\nu_0$ is not an eigenvalue of $\wh T_h$ and  $\mathbf{u}_0\in\b V_h$ with $\|\mathbf{u}_0\|_a=1$. Suppose that
\begin{eqnarray*}
&&dist(\mathbf{u}_0,M_h(\nu_k))\leq1/2, \\
&&|\nu_0-\nu_k|\leq \rho/4 , ~ |\nu_{j,h}-\nu_{j}|\leq\rho/4(j= k-1,k,k+q,j\neq0),
\end{eqnarray*}
where $\rho=\min_{\nu_j\neq\nu_k}|\nu_j-\nu_k|$.
Let $\mathbf{u}^s\in \b V_h,\mathbf{u}_k^h\in \b V_h$ satisfy
\begin{eqnarray*}
(\nu_0-\wh T_h)\mathbf{u}^s=\mathbf{u}_0,~~~\mathbf{u}_k^h=\mathbf{u}^s/\|\mathbf{u}^s\|_a.
\end{eqnarray*}
Then
\begin{eqnarray*}
dist(\mathbf{u}_k^h, {M}_h(\nu_k))\leq \frac{4}{\rho}\max_{k\le j\le k+q-1}|\nu_0-\nu_{j,h}|dist(\mathbf{u}_0,M_h(\nu_k)).
\end{eqnarray*}
\end{lemma}


Let $\delta_0$ and $\delta'_0$ be two positive constants such that

\begin{eqnarray}
&&\delta_0\le \min\{\overline{C}(4\sqrt{\wt\lambda_k})^{-1},\frac1 2\}, ~\delta'_0\le \frac{\lambda_k}{2}, ~
\frac{\delta'_0}{(\wt\lambda_k- \delta'_0 )\wt\lambda_k}\le\frac{\rho}{4},\label{3.7}\\
&&\4\3\delta_0^2<\wt\lambda_j,~\frac{  \4\3^2\delta_0^2   }{(\wt\lambda_j-\4\3\delta_0^2 )\wt\lambda_j}\le\frac{\rho}{4},~
j=k-1,k,k+q,j\neq0,\label{3.8}\\
&& (3+\3) \delta_0+
3\overline{C}^{-2}\4\delta_0^2 +\overline{C}^{-2}(3\lambda_k+2\overline{C}^2) \delta_0\le 1/2.\label{3.9}
\end{eqnarray}

In the coming theoretical analysis, in the step 3 of Scheme 3.1 and step 6 of Scheme  3.2 we introduce a new
auxiliary  variable $\mathbf{\widehat{u}}^{h_i}$ satisfying
$$\mathbf{\widehat{u}}^{h_i}=\frac{\mathbf{u}'}{\|\mathbf{u}'\|_A}.$$
Then it is clear that $\mathbf{u}^{h_i}=\frac{\mathbf{\widehat{u}}^{h_i}}{\|\mathbf{\widehat{u}}^{h_i}\|_a}$
and $\lambda^{h_i}=\frac{a(\mathbf{\wh u}^{h_{i}},\mathbf{\wh u}^{h_{i}})}{(\epsilon\mathbf{\wh u}^{h_{i}},\mathbf{\wh u}^{h_{i}})}.$

\b{Condition 3.1.} There exists $\overline{\b u}_k   \in M(\lambda_k)$ 
such that for some   $i\in\{1,2,\cdots,l\}$
\begin{eqnarray*}
&&\|\b{\widehat{u}}_k^{h_{l-1}}-\b{\overline{u}_k}\|_{A}\le \delta_0, ~\delta_{h_l} (\lambda_j ) \le \delta_0 (j = k - 1, k,k+q, j \neq 0),\\
&&|\lambda^{h_{i}}_k-\lambda_k| \le \delta'_0,
\end{eqnarray*}
where $\lambda^{h_{i}}_k$ and ($\lambda^{h_{l-1}}_k$,$\b{\widehat{u}}_k^{h_{l-1}}$)
 are   approximate eigenpairs corresponding to the eigenvalue $\lambda_k$
obtained by Scheme 3.1 or Scheme 3.2.

We are in a position to prove a critical theorem which establishes the error relation for approximate eigenpairs between
two adjacent iterations.
Our proof
shall   sufficiently make use of the relationship among the operators  $T$, $T_{h_l}$, $\wh T$, $\wh T_{h_l}$, $\wt T$ and
 $\wt T_{h_l}$, as
shown in Lemma 3.3, and   the proof method is an extension of that   in \cite{yang2}.

\begin{theorem}\label{l3}
Let $(\lambda_k^{h_l},\b{\widehat{u}}_k^{h_l})$ be an approximate eigenpair obtained by Scheme 3.1
or Scheme 3.2. Suppose Theorem 2.2 holds with  $\lambda=\lambda_{k-1},\lambda_{k},\lambda_{k+q}$,  and  Condition 3.1   holds with $i=l-1$ for Scheme 3.1 or with $i=i0,l-1$ for
Scheme 3.2. Let $\lambda_0=\lambda_k^{h_{l-1}}$ for
Scheme 3.1 or $\lambda_0=\lambda_k^{h_{i0}}$ for Scheme 3.2.   
Then there  exists $\mathbf{u}_k\in M(\lambda_k)$ such that
\begin{eqnarray}\label{s3.4}
  \|\mathbf{\widehat{u}}^{h_{l}}_k-\mathbf{u}_k\|_{A} \leq \frac{\0}{2}\Big(|\lambda_0-\lambda_k|(|\lambda^{h_{l-1}}_k-
  \lambda_k|+\|\mathbf{\widehat{u}}^{h_{l-1}}_k-\overline{\mathbf{u}}_k\|_{0,\epsilon})
  +\delta_{h_{l}}(\lambda_k)\Big),
\end{eqnarray}
where  $\0$ is independent of the mesh parameters  and the iterative times $l$.
\end{theorem}
 \begin{proof}
 Step 3 of Scheme 3.1 with $i = l$  is equivalent to: find $(\mathbf{u}_{h_l}',\sigma')\in U_{h_l}\times V_{h_l}$ such that
 \begin{eqnarray}\label{s3.20}
&&A(\mathbf{u}',\mathbf{v})-(\lambda_{0}+\overline{C}^2)A(\wh T_{h_l}\mathbf{u}',\mathbf{v})=A(\wh T_{h_l}\mathbf{u}^{h_{l-1}}_k,\mathbf{v}),~~~\forall\mathbf{v}\in \b V_{h_l},
\end{eqnarray}
and  $\mathbf{u}^{h_l}_k=\mathbf{u}'/\|\mathbf{u}'\|_a,~\mathbf{\wh u}^{h_l}_k=\mathbf{u}'/\|\mathbf{u}'\|_A$.
That is
\begin{eqnarray}\label{s3.22}\nonumber
((\lambda_{0}+\overline{C}^2)^{-1}-\wh T_{h_l})\mathbf{u'}=(\lambda_{0}+\overline{C}^2)^{-1}\wh T_{h_l}\mathbf{u}^{h_{l-1}}_k,~~~\mathbf{\b{\widehat{u}}}^{h_l}_k=\mathbf{u}'/\|\mathbf{u}'\|_A.
\end{eqnarray}
Denote
\begin{eqnarray}\nonumber
&&\nu_0=(\lambda_{0}+\overline{C}^2)^{-1},~~~\mathbf{u}_0=(\lambda^{h_{l-1}}_k+\overline{C}^2)\wh T_{h_{l}}\mathbf{\widehat{u}}^{h_{l-1}}_k
/\|(\lambda^{h_{l-1}}_k+\overline{C}^2)\wh T_{h_{l}}\mathbf{\widehat{u}}^{h_{l-1}}_k\|_{A},\\\nonumber
&& \mathbf{u}^s=(\lambda_{0}+\overline{C}^2)\mathbf{u}'/\|(\lambda^{h_{l-1}}_k+\overline{C}^2)\wh T_{h_{l}}\mathbf{u}^{h_{l-1}}_k\|_{A}, ~~~ \nu_{h_l}=1/\lambda^{h_{l }}_k .
\end{eqnarray}
Noting $\mathbf{u}_k^{h_{l-1}}=\mathbf{\widehat{u}}^{h_{l-1}}_k/\|\mathbf{\widehat{u}}^{h_{l-1}}_k\|_a$, then Step 3 of Scheme 3.1   is equivalent to:
\begin{eqnarray*}
(\nu_0-\wh T_{h_l})\mathbf{u}^s=\mathbf{u}_0,~~~\mathbf{\widehat{u}}^{h_l}_k=\mathbf{u}^s/\|\mathbf{u}^s\|_A.
\end{eqnarray*}
Noting $\|\overline{\mathbf{u}}_k\|_A\le1+\delta_0\le 3/2$, using Lemma 3.3 we derive from (\ref{3.4}) and (\ref{3.5})
\begin{eqnarray}\label{3.11}
&&\|(\lambda^{h_{l-1}}_k+\overline{C}^2)\wh T_{h_{l}}\mathbf{\overline{u}}_k-\overline{\mathbf{u}}_k\|_A=
  \|(\lambda_k+\overline{C}^2)\wh T\overline{\mathbf{u}}_k-(\lambda^{h_{l-1}}_k+\overline{C}^2)\wh T_{h_{l}}\mathbf{\overline{u}}_k\|_A \nonumber\\
&&~~~\le(\lambda_k+\overline{C}^2)\|(\wh T-\wh T_{h_{l}})\overline{\mathbf{u}}_k\|_A +\|(\lambda_k-\lambda^{h_{l-1}}_k)\wh T_{h_l}\overline{\mathbf{u}}_k\|_A \nonumber\\
&&~~~\le  \delta_{h_l}(\lambda_k)\|\overline{\mathbf{u}}_k\|_{A} +
  \overline{C}^{-1}|\lambda^{h_{l-1}}_k-\lambda_k|\|\overline{\mathbf{u}}_k\|_{0,\epsilon}\nonumber\\
 &&~~~\le  \frac{3}{2}\delta_{h_l}(\lambda_k) +
  \frac{3}{2}\overline{C}^{-2}|\lambda^{h_{l-1}}_k-\lambda_k|.
\end{eqnarray}

By (\ref{s3.1}), (\ref{3.5}) and (\ref{3.7}), we have
\begin{eqnarray}\nonumber
&&\|\mathbf{u}_{0}-\frac{\overline{\mathbf{u}}_k}{\|\overline{\mathbf{u}}_k\|_{A}}\|_{A}
  \leq 2
\|(\lambda^{h_{l-1}}_k+\overline{C}^2)\wh T_{h_{l}}\mathbf{\widehat{u}}^{h_{l-1}}_k-\overline{\mathbf{u}}_k\|_{A}\\\label{3.13}
&&~~~\le2(\|(\lambda^{h_{l-1}}_k+\overline{C}^2)\wh T_{h_{l}}\mathbf{\overline{u}}_k-\overline{\mathbf{u}}_k\|_A+\|(\lambda^{h_{l-1}}_k+\overline{C}^2)\wh T_{h_{l}}(\mathbf{\overline{u}}_k-\mathbf{\widehat{u}}^{h_{l-1}}_k)\|_{A} )\nonumber\\
&&~~~\le2\|(\lambda^{h_{l-1}}_k+\overline{C}^2)\wh T_{h_{l}}\mathbf{\overline{u}}_k-\overline{\mathbf{u}}_k\|_A+\overline{C}^{-1}(3\lambda_k+2\overline{C}^2)\|
\mathbf{\overline{u}}_k-\mathbf{\widehat{u}}^{h_{l-1}}_k\|_{0,\epsilon} .
\end{eqnarray}
We shall verify the conditions of Lemma \ref{l4}.
Recalling (\ref{2.23}), (\ref{3.6}) and (\ref{3.9}), the estimates (\ref{3.11}) and (\ref{3.13}) lead to
\begin{eqnarray}\nonumber
&&dist(\mathbf{u}_{0},{M}_{h_l}(\lambda_k))\leq \|\mathbf{u}_{0}-\frac{\overline{\mathbf{u}}_k}{\|\overline{\mathbf{u}}_k\|_{A}}\|_{A}
+dist(\frac{\overline{\mathbf{u}}_k}{\|\overline{\mathbf{u}}_k\|_{A}}, {M}_{h_l}(\lambda_k) )\\\nonumber
&&~~~\leq (3+\3) \delta_{h_l}(\lambda_k)+
3\overline{C}^{-2}|\lambda_k^{h_{l-1}}-\lambda_k| +\overline{C}^{-1}(3\lambda_k+2\overline{C}^2)\|
\mathbf{\overline{u}}_k-\mathbf{\widehat{u}}^{h_{l-1}}_k\|_{0,\epsilon}\\\nonumber
&&~~~\leq (3+\3) \delta_0+
3\overline{C}^{-2}\4\delta_0^{2} +\overline{C}^{-2}(3\lambda_k+2\overline{C}^2) \delta_0\\\label{3.13s}
&&~~~\le 1/2.
\end{eqnarray}
Due to Condition 3.1 we have from (\ref{3.7})
\begin{eqnarray}\nonumber
&&|\nu_{k}-\nu_0|=\frac{|\lambda_0-\lambda_{k}|}{|(\lambda_{0}+\overline{C}^2)\wt \lambda_{k}|}\le
\frac{\delta'_0}{(\wt\lambda_k- \delta'_0 )\wt\lambda_k}\le\frac{\rho}{4}.
\end{eqnarray}
Since by (\ref{3.6}), (\ref{2.22}) and (\ref{3.8}) we get
\begin{eqnarray}\label{3.14s}
 \wt\lambda_{j,h_l}\ge \wt\lambda_j-|\lambda_j-\lambda_{j,h_l}|
\ge  \wt\lambda_j-\4\3\delta_{h_l}^2(\lambda_j)
\ge  \wt\lambda_j-\4\3\delta_0^2>0
\end{eqnarray}
and then for $ j=k-1,k,k+q,j\neq0$
\begin{eqnarray}\nonumber
&&|\nu_{j}-\nu_{j,h_l}|=\frac{|\lambda_j-\lambda_{j,h_l}|}{|\wt\lambda_j \wt\lambda_{j,h_l}|}
\le \frac{  \4\3^2\delta_0^2   }{(\wt\lambda_j-\4\3\delta_0^2 )\wt\lambda_j}\le\frac{\rho}{4}.
\end{eqnarray}
Therefore the   conditions of Lemma \ref{l4} hold, and we have
\begin{eqnarray}\label{s3.23}
dist(\mathbf{\widehat{u}}_k^{h_l}, {M}_{h_l}(\lambda_k))\leq \frac{4}{\rho}\max_{k\le j\le k+q-1}|\nu_{j,h_{l}}-\nu_0|dist(\mathbf{u}_0,M_{h_l}(\lambda_k)).
\end{eqnarray}

Applying (\ref{3.14s}), (\ref{3.6}) and (\ref{2.22}) we have for $j=k,k+1,\cdots,k+q-1$
\begin{eqnarray}\nonumber
&&|\nu_{j,h_{l}}-\nu_0|=\frac{|\lambda_0-\lambda_{j,h_l}|}{|(\lambda_{0}+\overline{C}^2) \wt\lambda_{j,h_l}|}\le
\frac{|\lambda_0-\lambda_k|+|\lambda_k-\lambda_{j,h_l}|}{(\wt\lambda_k- \delta'_0 ) (\wt\lambda_k-\4\3\delta_0^2)}\\\label{3.15}
&&~~~\le \frac{|\lambda_0-\lambda_k|+\4\3\delta_{h_l}^2(\lambda_k)}{(\wt\lambda_k- \delta'_0 ) (\wt\lambda_k-\4\3\delta_0^2)}.
\end{eqnarray}
 Substituting (\ref{3.13s}) and (\ref{3.15}) into (\ref{s3.23}), we have
\begin{eqnarray}\nonumber
&&dist(\mathbf{\widehat{u}}^{h_l}_k, {M}_{h_l}(\lambda))\leq \frac{4}{\rho}\Big(
\frac{|\lambda_0-\lambda_k|+\4\3\delta_{h_l}^2(\lambda_k)}{|(\wt\lambda_k- \delta'_0 ) (\wt\lambda_k-\4\3\delta_0^2)|}\Big) \times\\\nonumber
&&~~~\big((3+\3) \delta_{h_l}(\lambda_k)+
3\overline{C}^{-2}|\lambda^{h_{l-1}}_k-\lambda_k| +\overline{C}^{-1}(3\lambda_k+2\overline{C}^2)\|
\mathbf{\overline{u}}_k-\mathbf{\widehat{u}}^{h_{l-1}}_k\|_{0,\epsilon}\big).\\\label{3.16}
\end{eqnarray}
Let $ \mathbf{u}_{j,{h}}$ be the eigenfunction corresponding to $\lambda_{j,{h}}$ such that
$\{\mathbf{u}_{j,{h}}\}_{j=k}^{k+q-1} $ constitutes an orthonormal basis of $M_{h}(\lambda)$ in the sense of norm $\|\cdot\|_A$.
Let $ \mathbf{u}^* =  \sum\limits_{j=k}^{k+q-1}A(\mathbf{\widehat{u}}^{h_l}_k,\mathbf{u}_{j,h_l}) \mathbf{u}_{j,h_l}$ then $
 \| \mathbf{\widehat{u}}^{h_{l}}_k-\mathbf{u}^*\|_{A}=  dist(\mathbf{\widehat{u}}^{h_l}_k, {M}_{h_l}(\lambda_k))$.

 From Theorem 2.2, we know there exists $\{\mathbf{u}_j^0\}_k^{k+q-1}\subset M(\lambda_k)$ such that $\mathbf{u}_{j,h_l}-\mathbf{u}_j^0$ satisfies (\ref{2.22})﹛ and it holds
by taking $\mathbf{u}_k =  \sum\limits_{j=k}^{k+q-1}A(\mathbf{\widehat{u}}^{h_l}_k,\mathbf{u}_{j,h_l})\mathbf{u}_j^0$
\begin{eqnarray}\nonumber
\|\mathbf{u}_k-\mathbf{u}^*\|_{A} &=& \|\sum\limits_{j=k}^{k+q-1}A(\mathbf{\widehat{u}}^{h_l}_k,\mathbf{u}_{j,h_l})(\mathbf{u}_j^0-\mathbf{u}_{j,h_l})\|_{A}\\\nonumber
&\leq& (\sum\limits_{j=k}^{k+q-1}\|\mathbf{u}_j^0-\mathbf{u}_{j,h_l}\|^2_{A})^{1/2}\\\label{s3.34}
&\leq&\3\sqrt q  \delta_{h_l}(\lambda_k).
\end{eqnarray}
Therefore,  summing up (\ref{3.16}) and (\ref{s3.34}), we know 
there exists a positive constant $\0\ge\3$ that is independent of mesh parameters
and $l$ such that (\ref{s3.4})  holds.
\end{proof}

\indent {\bf Condition 3.2.}~~For any given $\varepsilon\in (0,2)$, there exist
 $t_{i}\in (1,3-\varepsilon]$ such that $\delta_{h_{i}}(\lambda_{k})=\delta^{t_{i}}_{h_{i-1}}(\lambda_{k})$
 and $\delta_{h_{i}}(\lambda_{k})\to 0~(i\to\infty)$.\\

\indent Condition 3.2 is easily satisfied. For example, for smooth solution, by using
the uniform mesh, let $h_0 =\sqrt{2}/8$, $h_1 =\sqrt{2}/32$, $h_2 =\sqrt{2}/64$ and $h_3=\sqrt{2}/128$, we have $h_i = h_{i-1}^{t_i}$,
i.e., $\delta_{h_i} = \delta_{h_{i-1}} ^{t_i} $, where $t_1\approx 1.80$,
$t_2 \approx 1.22$, $t_3 \approx 1.18$. For non-smooth solution, the condition could be satisfied
when the local refinement is performed near   the singular points.

\begin{theorem}
Let $(\lambda_{k}^{h_{l}}, \b{\widehat{u}}_{k}^{h_{l}})$ be the approximate eigenpairs
obtained by Scheme 3.1. Suppose  Condition
3.2 holds. Then there exist  $\b u_{k}\in M(\lambda_{k})$ and $H_0>0$ such that if $H\le H_0$ then
\begin{eqnarray}\label{s4.15}
&&\|\b{\widehat{u}}_{k}^{h_{l}}-\b u_{k}\|_{A}\leq  C_{0}
\delta_{h_{l}}(\lambda_{k}),\\\label{s4.16}
&&|\lambda_{k}^{h_{l}}-\lambda_{k}|\leq C_{4} C_{0}^{2}
\delta_{h_{l}}^{2}(\lambda_{k}),~~~ l\ge1.
\end{eqnarray}
\end{theorem}
\begin{proof}
The proof is completed by using induction and Theorem 3.5 with
$\lambda_{0}=\lambda_{k}^{h_{l-1}}$.
 Noting that $\delta_H(\lambda_k)\rt 0$ as $H\rt 0$,  there exists $H_0>0$ such that if $H<H_0$ then Theorem 2.2
 holds for $\lambda=\lambda_{k-1},\lambda_{k},\lambda_{k+q}$  and
\begin{eqnarray*}
&&\0\delta_H(\lambda_k)\le \delta_0,~\4\0^2\delta_H^2(\lambda_k)\le \delta'_0,~
\delta_H(\lambda_j)\le \delta_0,(j=k-1,k,k+q,j\ne0),\\
&&C_{4}^{2}C_{0}^{4}\delta_{H}^{1+\varepsilon}(\lambda_{k})+C_{4}C_{0}^{3}\overline{C}^{-1}
\delta_{H}^{\varepsilon}(\lambda_{k})\leq 1.
\end{eqnarray*}
 When $l=1$,
$(\lambda_{k}^{h_{l-1}}, \b{\widehat{u}}_{k}^{h_{l-1}})=(\lambda_{k,H}, \b{\widehat{u}}_{k,H})$,
from (\ref{2.24}) and (\ref{3.6}) we know that
there exists $\overline{\b u}_k\in M(\lambda_{k})$ such that
\begin{eqnarray*}
&&\|\b{\widehat{u}}_{k,H}-\overline{\b u}_k\|_{A}\leq C_{3} \delta_{H}(\lambda_{k}),\\
&&|\lambda_{k,H}-\lambda_{k}|\leq
C_{4}C_{3}^{2}\delta_{H}^{2}(\lambda_{k}).
\end{eqnarray*}
Then $\|\b{\widehat{u}}_{k}^{h_{0}}-\overline{\b u}_k\|_{A}\leq \3\delta_{H}(\lambda_{k})\leq
\delta_{0}$, $|\lambda_{k}^{h_{0}}-\lambda_{k}|\leq
C_{4}\3^{2}\delta_{H}^{2}(\lambda_{k})\leq \delta'_{0}$ and
$\delta_{h_{1}}(\lambda_{j})\leq \delta_{0}$ $(j=k-1,k,k+q,j\ne0)$, i.e.,
Condition 3.1 holds for $l=1$. Thus, by Theorem 3.5 and $C_{3}\leq C_{0}$ we
get
\begin{eqnarray*}
&&\|\b{\widehat{u}}_{k}^{h_{1}}-\b{u}_{k}\|_{A}\leq
\frac{\0}{2}\{C_{4}^{2}C_{0}^{4}\delta_{H}^{4}(\lambda_{k})
+C_{4}C_{0}^{3}\overline{C}^{-1}\delta_{H}^{3}(\lambda_{k})+\delta_{h_{1}}(\lambda_{k})\}\nonumber\\
&&~~~\leq
\frac{\0}{2}\{C_{4}^{2}C_{0}^{4}\delta_{H}^{4-t_{1}}(\lambda_{k})+C_{4}C_{0}^{3}
\overline{C}^{-1}\delta_{H}^{3-t_{1}}(\lambda_{k})+1\} \delta_{h_{1}}(\lambda_{k})\nonumber\\
&&~~~\leq
\frac{\0}{2}\{C_{4}^{2}C_{0}^{4}\delta_{H}^{1+\varepsilon}(\lambda_{k})+C_{4}C_{0}^{3}
\overline{C}^{-1}\delta_{H}^{\varepsilon}(\lambda_{k})+1\} \delta_{h_{1}}(\lambda_{k}),
\end{eqnarray*}
where we have used  the fact  $3-t_{1}\geq \varepsilon$. This yields (\ref{s4.15}) and (\ref{s4.16})
for $l=1$. 
Suppose that Theorem 3.6 holds for $l-1$, i.e.,   there exists
$\overline{\b u}_k\in M(\lambda_{k})$ such that
\begin{eqnarray*}
&&\|\b{\widehat{u}}_{k}^{h_{l-1}}-\overline{\b u}_k\|_{A}\leq  C_{0}\delta_{h_{l-1}}(\lambda_{k}),\\
&&|\lambda_{k}^{h_{l-1}}-\lambda_{k}|\leq C_{4} C_{0}^{2}
\delta_{h_{l-1}}^{2}(\lambda_{k}),
\end{eqnarray*}
then $\|\b{\widehat{u}}_{k}^{h_{l-1}}-\overline{\b u}\|_{A}\leq \delta_{0}$,
$|\lambda_{k}^{h_{l-1}}-\lambda_{k}|\leq \delta'_{0}$
 and $\delta_{h_{l}}(\lambda_{j})\leq \delta_{0}$ $(j=k-1,k,k+q,j\ne0)$,
and the conditions of Theorem 3.5 hold. Therefore, for $l$, by
(\ref{s3.4}) we deduce
\begin{eqnarray*}
&&\|\b{\widehat{u}}_{k}^{h_{l}}-\b u_{k}\|_{A}\leq
\frac{\0}{2}\{C_{4}^{2}C_{0}^{4}\delta_{h_{l-1}}^{4}(\lambda_{k})
+C_{4}C_{0}^{3}\overline{C}^{-1}\delta_{h_{l-1}}^{3}(\lambda_{k})
+\delta_{h_{l}}(\lambda_{k})\}\nonumber\\
&&~~~\leq
\frac{\0}{2}\{C_{4}^{2}C_{0}^{4}\delta_{h_{l-1}}^{4-t_{i}}(\lambda_{k})
+C_{4}C_{0}^{3}\overline{C}^{-1}\delta_{h_{l-1}}^{3-t_{i}}(\lambda_{k})
+1\}\delta_{h_{l}}(\lambda_{k})\nonumber\\
&&~~~\leq
\frac{\0}{2}\{C_{4}^{2}C_{0}^{4}\delta_{H}^{1+\varepsilon}(\lambda_{k})
+C_{4}C_{0}^{3}\overline{C}^{-1}\delta_{H}^{\varepsilon}(\lambda_{k})
+1\}\delta_{h_{l}}(\lambda_{k})\nonumber\\
&&~~~\leq C_{0} \delta_{h_{l}}(\lambda_{k}),
\end{eqnarray*}
i.e., (\ref{s4.15})  are valid. And from
(\ref{s4.15}) and (\ref{3.6}) we get (\ref{s4.16}). This ends the proof.
\end{proof}

 \indent {\bf Condition 3.3.}~
There exist $\beta_{0}\in (0,1)$ and
$\beta_{i}\in [\beta_{0}, 1)$ ($i=1,2,\cdots$) such that
$\delta_{h_{i}}(\lambda_{k})=\beta_{i}\delta_{h_{i-1}}(\lambda_{k})$
 and $\delta_{h_{i}}(\lambda_{k})\to 0~(i\to\infty)$.\\

\indent {\bf Remark 3.2.}
 Note that if Condition 3.3 is valid, Condition 3.2 holds for $H$ properly small;
however, the inverse is not true. So in Theorem 3.6, (\ref{s4.15}) and (\ref{s4.16}) still hold if we replace
Condition 3.2 with  Condition 3.3.\\

\begin{theorem}
Let $(\lambda_{k}^{h_{l}}, \b{\widehat{u}}_{k}^{h_{l}})$  be an approximate
eigenpair  obtained by Scheme 3.2. Suppose that  Condition 3.2 holds for $i\le i0$
and Condition 3.3 holds for $i> i0$. Then there exist  $\b u_{k}\in M(\lambda_{k})$ and $H_0>0$ such that if $H\le H_0$ then
\begin{eqnarray}\label{s4.18}
&&\|\b{\widehat{u}}_{k}^{h_{l}}-\b u_{k}\|_{A}\leq
C_{0}\delta_{h_{l}}(\lambda_{k}),\\\label{s4.19}
&&|\lambda_{k}^{h_{l}}-\lambda_{k}|\leq
C_{4}C_{0}^{2}\delta_{h_{l}}^{2}(\lambda_{k}),~~~l> i0.
\end{eqnarray}
\end{theorem}
\begin{proof}
The proof is completed by using induction. Noting    $\delta_H(\lambda_k)\rt 0$ as $H\rt 0$,
there exists $H_0>0$ such that if $H<H_0$  then  Theorems 2.2 holds for $\lambda=\lambda_{k-1},\lambda_{k},\lambda_{k+q}$,
 Theorem 3.6  holds  and
\begin{eqnarray*}
&&\0\delta_H(\lambda_k)\le \delta_0,~\4\0^2\delta_H^2(\lambda_k)\le \delta'_0,~
\delta_H(\lambda_j)\le \delta_0,(j=k-1,k,k+q,j\ne0),\\
&&C_{4}^{2}C_{0}^{4}\delta_{H}^{3}(\lambda_{k})\frac{1}{\beta_{0}}
+C_{4}\frac{C_{0}^3}{\overline{C}}\delta_{H}^{2}(\lambda_{k})\frac{1}{\beta_{0}}\le 1.
\end{eqnarray*}

 From Theorem 3.6 we know that  when $l=i0,i0+1$ there exists $\b u_{k}\in M(\lambda_{k})$ such that
\begin{eqnarray*}
&&\|\b{\widehat{u}}_{k}^{h_{l}}-\b u_{k}\|_{A}\leq
C_{0}\delta_{h_{l}}(\lambda_{k}),\\
&& |\lambda_{k}^{h_{l}}-\lambda_{k}|\leq
C_{4}C_{0}^{2}\delta_{h_l}^{2}(\lambda_{k}).
\end{eqnarray*}
Suppose  Theorem 3.7 holds for $l-1$, i.e., there exists
$\b{\overline{u}}_k\in M(\lambda_{k})$ such that
\begin{eqnarray*}
&&\|\b{\widehat{u}}_{k}^{h_{l-1}}-\b{\overline{u}}_k\|_{A}\leq
C_{0}\delta_{h_{l-1}}(\lambda_{k}),\\
&&|\lambda_{k}^{h_{l-1}}-\lambda|\leq
C_{4}C_{0}^{2}\delta_{h_{l-1}}^{2}(\lambda_{k}).
\end{eqnarray*}
Then the conditions of Theorem 3.5 hold, therefore, for $l$, observing that in (\ref{s3.4})
$\|\b{\widehat{u}}_{k}^{h_{l-1}}-\b{\overline{u}}_k\|_{0,\epsilon}$ can be replaced by
$\overline{C}^{-1}\|\b{\widehat{u}}_{k}^{h_{l-1}}-\b{\overline{u}}_k\|_{A}$, we deduce
\begin{eqnarray*}
&&\|\b{\widehat{u}}_{k}^{h_{l}}-\b u_{k}\|_{a}\leq
\frac{C_{0}}{2}\Big\{C_{4}^{2}C_{0}^{4}\delta_{h_{i0}}^{2}(\lambda_{k})\Big(\delta_{h_{l-1}}^{2}(\lambda_{k})
+\frac{C_{0}}{\overline{C}}\delta_{h_{l-1}}(\lambda_{k})\Big)
+\delta_{h_{l}}(\lambda_{k})\Big\}\nonumber\\
&&~~~\leq
\frac{C_{0}}{2}\{C_{4}^{2}C_{0}^{4}\delta_{h_{i0}}^{2}(\lambda_{k})\delta_{h_{l-1}}(\lambda_{k})\frac{1}{\beta_{0}}
+C_{4}\frac{C_{0}^3}{\overline{C}}\delta_{h_{i0}}^{2}(\lambda_{k})\frac{1}{\beta_{0}}
+1\}\delta_{h_{l}}(\lambda_{k}),
\end{eqnarray*}
noting that
$\delta_{h_{l-1}}(\lambda_{k})\leq
\delta_{h_{i0}}(\lambda_{k})\leq\delta_{H}(\lambda_{k})$, we get
(\ref{s4.18})
immediately.
(\ref{s4.19}) can be obtained from (\ref{s4.18}) and (\ref{3.6}).
The proof is completed.
\end{proof}

\indent{\bf Remark 3.3.}
The error estimates (\ref{s4.15}) and (\ref{s4.18}) for $\b{\wh u}^{h_l}_k$ can lead to the error estimates for $\b u^{h_l}_k$. In fact,
under the conditions of Theorem 3.6 or Theorem 3.7, we have
$$\|\b u_k\|_{A}\ge \| \b{\wh u}^{h_l}_k\|_A-\0\delta_{h_l}(\lambda_k)\ge1-\delta_0\ge1/2,$$
then  $\|\b u_k\|_a=\frac{\sqrt{\lambda_k}}{\sqrt{\wt\lambda_k}}\|\b u_k\|_{A}\ge  \frac{\sqrt{\lambda_k}}{2\sqrt{\wt\lambda_k}}$.
We further assume $\delta_0\le\frac{\sqrt{\lambda_k}}{4\sqrt{\wt\lambda_k}}$ then

\begin{eqnarray*}
 \| \b{\wh u}^{h_l}_k\|_a\ge\| \b{u}_k\|_a-\0\delta_{h_l}(\lambda_k)\ge \frac{\sqrt{\lambda_k}}{2\sqrt{\wt\lambda_k}}-\delta_0
\ge\frac{\sqrt{\lambda_k}}{4\sqrt{\wt\lambda_k}}.
\end{eqnarray*}
Therefore we derive from (\ref{s4.15}) or (\ref{s4.18})
\begin{eqnarray*}
&&\| \b{u}^{h_l}_k -\frac{\b u_k}{\|\b u_k\|_a}\|_A\leq
\frac{\|\b{\widehat{u}}^{h_l}_k-\b u_k\|_A\|\b u_k\|_a+\|\b{\widehat{u}}^{h_l}_k-\b u_k\|_a\|\b u_k\|_A}{\|\b{\widehat{u}}^{h_l}_k\|_a\|\mathbf{u}_k\|_a}\\
&&~~~\leq \frac{(\sqrt{\lambda_k}+\sqrt{\wt\lambda_k})\|\b{\widehat{u}}^{h_l}_k-\b u_k\|_A}{\|\b{\widehat{u}}^{h_l}_k\|_a\sqrt{\lambda_k}}
\leq \frac{4(\sqrt{\lambda_k\wt\lambda_k}+\wt\lambda_k)\0\delta_{h_l}(\lambda_k)}{\lambda_k},
\end{eqnarray*}
i.e., $\b{u}^{h_l}_k$ has the same convergence order as $\b{\widehat{u}}^{h_l}_k$ in the sense of $\|\cdot\|_A$.

\section{Numerical experiment}

\indent In this section, we will report  several numerical experiments for solving the Maxwell eigenvalue
problem by multigrid Scheme 3.2 using the lowest order edge element
to validate our theoretical results.  We use MATLAB 2012a to compile our program codes  and adopt the data structure of finite elements in  the package
of iFEM \cite{chen} to generate and refine the meshes.
We use the sparse solver $eigs(A,B,k,'sm')$
to solve (\ref{s2.1}) for   $k$ lowest eigenvalues.
 In our tables 4.1-4.6   we use the
notation $\lambda_{k}^{h_i}$ to denote the numerical eigenvalue approximating    $\lambda_k$
   obtained by multigrid methods at $i$th iteration   on the mesh $\pi_{h_i}$ (with number of
   degree of freedom $N^{h_i}$), and $R$ to denote the convergence rate with respect to $Dof^{-1/3}$ where $Dof$ is  the number of degrees of freedom.
 For comparative purpose, $\lambda_{k,h_j}$ denotes the numerical eigenvalue approximating $\lambda_k$
   computed by the direct solver $eigs$ on the mesh $\pi_{h_j}$.\\
\\
\indent {\em  Example 4.1.} Consider the Maxwell eigenvalue problem with $\mu=\epsilon=I$ on
the unit cube $\Omega=(-\frac{1}{2},\frac{1}{2})^3$. 
We use Scheme 3.2 with $i0=0$ to compute the  eigenvalues $\lambda_1=2\pi^2,
\lambda_4=3\pi^2,\lambda_6=5\pi^2$ (of multiplicity 3, 2 and 6 respectively).
 The numerical results are shown in Table  4.1, which indicates that numerical eigenvalues obtained by multigrid methods achieve the optimal convergence rate $R\approx2$.\\

\begin{table}
\caption{
 Example 4.1: $\Omega=(-\frac{1}{2},\frac{1}{2})^3,~\mu=\epsilon=I$  with $i0=0$.}
\begin{center} \footnotesize
\begin{tabular}{ccccccccccc}\hline
$i$&$N^{h_i}$ & $\lambda_{1}^{h_i}$ & $\lambda_{1,h_i}$&$R$&$\lambda_{4}^{h_i}$ & $\lambda_{4,h_i}$ &$R$&
$\lambda_{6}^{h_i}$&$\lambda_{6,h_i}$&$R$\\\hline
0&343&  19.5047&  19.5047& & 30.4126&  30.4126&& 45.8124&  45.8124 \\
1&3032&   19.6932&  19.6933&2.24 & 29.8349&   29.8356&1.75&  48.6276&48.6649&2.19
\\
2&26416 &  19.7284& 19.7284& 2.01 & 29.6667& 29.6670&1.89& 49.1714&49.1846&1.95\\
3&220256 &    19.7365& --&1.96  &29.6234& --&1.95 & 49.3040&--&1.97  \\
 \hline
\end{tabular}
\end{center}
\end{table}

\begin{table}
\caption{Example 4.2:
 $\Omega=((-1,1)^2\backslash (-1,0]^2)\times (0,1)$, $\mu=\epsilon=I$ with $i0=1$.}
\begin{center} \footnotesize
\begin{tabular}{ccccccccccc}\hline
$i$&$N^{h_i}$ & $\lambda_{1}^{h_i}$ & $\lambda_{1,h_i}$&$R$&$\lambda_{2}^{h_i}$ & $\lambda_{2,h_i}$&$R$ &
$\lambda_{3}^{h_i}$&$\lambda_{3,h_i}$&$R$\\\hline
0&99&      9.8975&    9.8975&&     10.8525&  10.8525&&  12.3382&  12.3382&\\
1&1028&    9.7980&   9.7981&0.63&10.9927& 10.9930&0.43&13.2728& 13.2881&2.69
\\
2&20200 &  9.6959&   9.7159&1.04&  11.2580&  11.2421&1.41& 13.3727& 13.3763&1.45\\
3&345680 &  9.6582& --&1.17  & 11.3330& --&2.08  &  13.4043 &--&4.00  \\
 \hline
\end{tabular}
\end{center}
\end{table}

\begin{table}
\caption{Example 4.2:
 $\Omega=((-1,1)^2\backslash (-1,0]^2)\times (0,1)$, $\epsilon=I$ and $\mu$ is as
 in (\ref{4.1}) with $i0=1$.}
\begin{center} \footnotesize
\begin{tabular}{cccccccccc}\hline
$i$&$N^{h_i}$ & $\lambda_{1}^{h_i}$ & $\lambda_{1,h_i}$&$\lambda_{2}^{h_i}$ & $\lambda_{2,h_i}$ &
$\lambda_{3}^{h_i}$&$\lambda_{3,h_i}$\\\hline
0&99&    2.7998&    2.7998 &    3.4243&    3.4243&     4.1146&  4.1146 \\
1&1028&    3.0331&    3.0467&  3.8965&     3.9895&  4.7137&   4.8214\\
2&20200 &  2.9313&    2.9311& 3.8243&   3.8244&   4.5235&   4.5234\\
3&345680 &   2.9138& --  &  3.8030& -- &  4.4881&--  \\
 \hline
\end{tabular}
\end{center}
\end{table}

\begin{table}
\caption{Example 4.2:
 $\Omega=((-1,1)^2\backslash (-1,0]^2)\times (0,1)$, $\mu=\epsilon=I$ with $i0=0$.}
\begin{center} \footnotesize
\begin{tabular}{cccccccccc}\hline
$k$&Coarse & Refine 1&Refine 2\\\hline
&2269&      31526&    298690\\
1&9.6932&   9.6130($R$=0.79)&   9.6364($R$=2.79)\\
2&11.1805&    11.3017($R$=1.52)&   11.3386($R$=2.52)\\
 \hline
\end{tabular}
\end{center}
\end{table}

\begin{table}
\caption{Example 4.3: $\Omega=(-\frac{1}{2},\frac{1}{2})\times (0,0.1)\times (-\frac{1}{2},\frac{1}{2})$, $\mu=I$; if $x_3>0$ then $\epsilon=2I$ otherwise $\epsilon=I$; $i0=0$.}
\begin{center} \footnotesize
\begin{tabular}{ccccccccccc}\hline
$i$&$N^{h_i}$ & $\lambda_{1}^{h_i}$ & $\lambda_{1,h_i}$&$R$&$\lambda_{2}^{h_i}$ & $\lambda_{2,h_i}$ &$R$&
$\lambda_{3}^{h_i}$&$\lambda_{3,h_i}$\\\hline
0&361&    12.52156&   12.52156 &&   29.5523&   29.5523&&    35.5836&   35.5836 \\
1&4202&   12.51452&   12.51452&0.45& 29.6070&   29.6070&1.04& 35.8558&  35.8559\\
2&39124 & 12.51586&   12.51586&0.84 & 29.6383& 29.6383&1.94& 35.9460& 35.9461\\
3&693928 &  12.51592& --&  &29.6458& --&1.55 &  35.9704&--  \\
 \hline
\end{tabular}
\end{center}
\end{table}

\begin{table}
\caption{
Example 4.4: $\Omega=(-1,1)^3\backslash [-\frac{1}{2},\frac{1}{2}]^3,~\mu=\epsilon=I$, $i0=0$.}
\begin{center} \footnotesize
\begin{tabular}{cccccccccc}\hline
$i$&$N^{h_i}$ & $\lambda_{2}^{h_i}$ & $\lambda_{2,h_i}$&$\lambda_{5}^{h_i}$ & $\lambda_{5,h_i}$ &
$\lambda_{7}^{h_i}$&$\lambda_{7,h_i}$\\\hline
0&6230 &  2.1379&  2.1379 & 6.1583& 6.1583 &  6.2501& 6.2501 \\
1&47320& 2.2582&  2.2588  & 6.2049&  6.2049  &  6.8060& 6.8128\\
2&808496 & 2.3163   & --  & 6.2315  & -- & 7.0723&--  \\
 \hline
\end{tabular}
\end{center}
\end{table}

\begin{figure}
  \centering
  \includegraphics[width=2.3in]{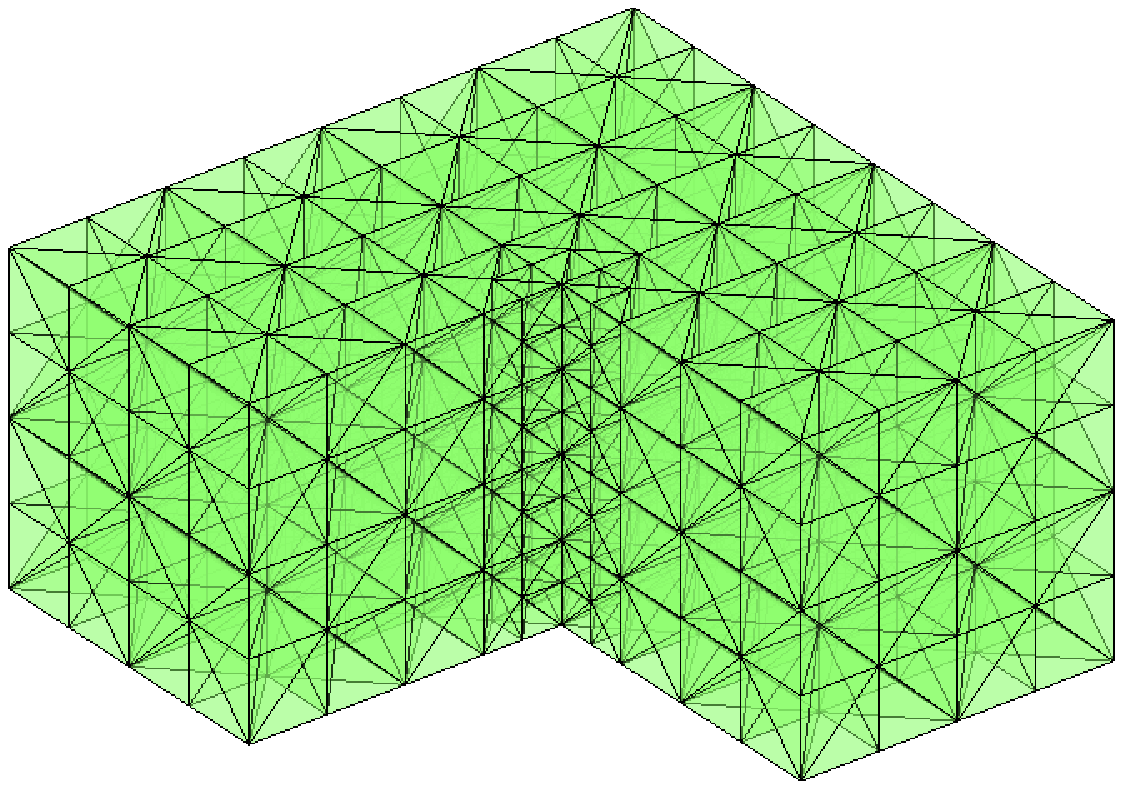}
  \includegraphics[width=2.3in]{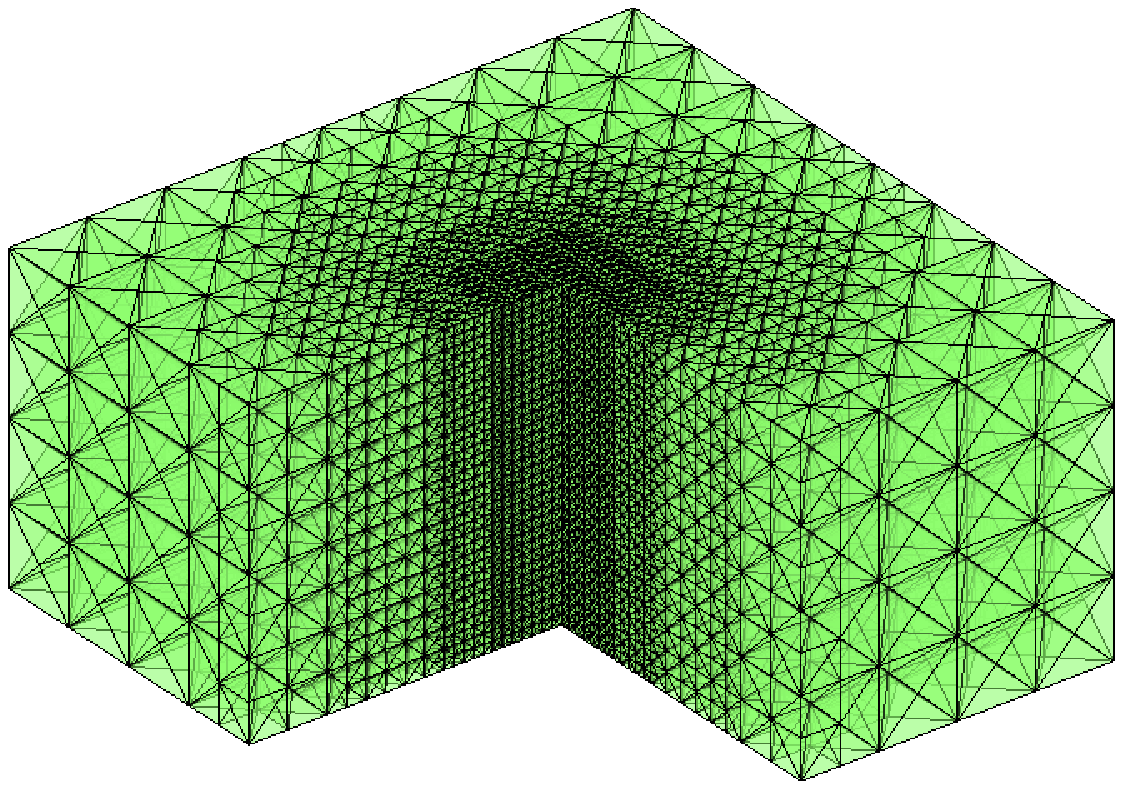}
  \caption{Coarse mesh(at left) and the second locally refined mesh (at right)}
\end{figure}

\indent{\em Example 4.2.} Consider the  eigenvalue problem with $\epsilon=I$ and $\mu=I$ or
\begin{eqnarray}\label{4.1}
\mu=\left(\begin{array}{ccc}
      2&1-2j&-j\\
      1 + 2 j&4&j\\
       j&-j&5
     \end{array}\right)
\end{eqnarray}
  on the thick L-shaped $\Omega=((-1,1)^2\backslash (-1,0]^2)\times (0,1)$.
  When $\mu=\epsilon=I$ $\lambda_1\approx9.6397$, $\lambda_2\approx11.3452$
 and $\lambda_3\approx13.4036$ (see \cite{dauge2}).  We use Scheme 3.2
  with $i0=1$ to compute the lowest three eigenvalues for both cases. The numerical results are shown
in Tables 4.2-4.3. From Table 4.2 we see that the   eigenvalue errors  obtained by Scheme 3.2
after 2nd iteration are respectively
0.019, 0.012 and 7.0e-04, which indicates the accuracy of  the  lowest two eigenvalues is affected
by the singularity of the associated eigenfunctions  in the directions
perpendicular to the reentrant edge and the convergence rate $R$ is usually less than 2. Alternately, we adopt the meshes  locally refined towards
the reentrant edge (see Figure 4.1) to perform the
iterative procedure. And the associated numerical results are listed in Table 4.4,  which implies
the errors of $\lambda^{h_2}_1$ and $\lambda^{h_2}_2$ are significantly
 decreased to 0.0033 and 0.0066 respectively  with less degrees of freedom.
\\\\
\indent{\em Example 4.3.} Consider the Maxwell eigenvalue problem with $\Omega=(-\frac{1}{2},\frac{1}{2})\times (0,0.1)\times (-\frac{1}{2},\frac{1}{2})$
 where $\mu=I$ and if $x_3>0$  then $\epsilon=2I$ otherwise $\epsilon=I$.
This is a practical problem in engineering computed in \cite{chatterjee,milan}.
We use Scheme 3.2
  with $i0=1$ to compute the three lowest  eigenvalues for both cases. The numerical results are shown
in Table 4.5.  The relatively accurate eigenvalues reported in \cite{chatterjee} are respectively  $3.538^2(12.5174)$, $5.445^2(29.6480)$
and  $5.935^2(35.2242)$. Using them as the reference values, the relative errors of numerical eigenvalues after 3rd iteration
are respectively 6.0877e-05, 3.7524e-05 and  0.0105. Obviously we get the good approximations of
the eigenvalues $\lambda_1$ and $\lambda_2$. Regarding the computation for $\lambda_3$, we refer to
 Table II in \cite{milan} whose relative error for computing $\lambda_3$ is 0.0107 using a higher order
 method. This is a computational result very close to ours.
Hence we think our method is also efficient for solving the problem.
\\\\
\indent{\em Example 4.4.} Consider the Maxwell eigenvalue problem with $\mu=\epsilon=I$ and $\Omega=(-1,1)^3\backslash [-\frac{1}{2},\frac{1}{2}]^3$. In this example, we capture a physical zero eigenvalue on a
coarse mesh with number of degrees of freedom 6230, i.e., $\lambda_{1,H}=1.9510e-12$.
We use Scheme 3.2
  with $i0=0$ to compute the   eigenvalues $\lambda_{2}$(of multiplicity 3), $\lambda_5$(of multiplicity 2) and $\lambda_7$(of multiplicity 3). The numerical results are shown
in Table 4.6. Note that the coarse mesh
seems slightly ``fine''. This is because we would like to capture all  information of the lowest
 eight eigenvalues (some of them would not be captured on a very coarse mesh). This is an example of handling the problem
in a cavity with two disconnected boundaries. For more numerical examples of the cavity with disconnected
boundaries,
we refer the readers to the work of \cite{jiang}.
\\\\
\indent{\bf Acknowledgment} The author wishes to thank Prof. Yidu Yang for many
valuable comments on this paper.

\end{document}